\documentclass[a4paper,oneside,9pt]{article}%
\usepackage{makeidx}
\usepackage[english]{babel}
\usepackage{amsmath}
\usepackage{amsfonts}
\usepackage{amssymb}
\usepackage{stmaryrd}
\usepackage{graphicx}
\usepackage{mathrsfs}
\usepackage{dsfont}
\usepackage{euscript}
\usepackage[colorlinks,linkcolor=red,anchorcolor=blue,citecolor=blue,urlcolor=blue]{hyperref}
\usepackage[symbol*,ragged]{footmisc}

\allowdisplaybreaks[4]
\providecommand{\U}[1]{\protect\rule{.1in}{.1in}}
\providecommand{\U}[1]{\protect \rule{.1in}{.1in}}

\newtheorem{theorem}{Theorem}[section]

\newtheorem{corollary}[theorem]{Corollary}

\newtheorem{lemma}[theorem]{Lemma}

\newtheorem{proposition}[theorem]{Proposition}

\newenvironment{proof}[1][Proof]{\noindent \textbf{#1.} }{\  \rule{0.5em}{0.5em}}
\numberwithin{equation}{section}

\begin{document}

\title{On Waring--Goldbach Problem for Squares, \\ Cubes and Higher Powers }
%  \author{Min Zhang\footnotemark   \vspace*{-5mm} \\
   %  \small Department of Mathematics, China University of Mining and Technology \vspace*{-5mm} \\
  %  \small  Beijing 100083, P. R. China  }

\author{Min Zhang\footnotemark[1] \,\,\,\,\,  \& \,\,\, Jinjiang Li\footnotemark[2]
                    \vspace*{-4mm} \\
    \small School of Applied Science, Beijing Information Science and Technology University\footnotemark[1]
                    \vspace*{-4mm} \\
     \small  Beijing 100192, P. R. China
                     \vspace*{-4mm}  \\
    \small Department of Mathematics, China University of Mining and Technology\footnotemark[2]
                    \vspace*{-4mm}  \\
     \small  Beijing 100083, P. R. China  \vspace*{-4mm}  }

\footnotetext[2]{Corresponding author. \\
    \quad\,\, \textit{ E-mail addresses}:
     \href{mailto:min.zhang.math@gmail.com}{min.zhang.math@gmail.com} (M. Zhang),
     \href{mailto:jinjiang.li.math@gmail.com}{jinjiang.li.math@gmail.com} (J. Li).  }

\date{}
\maketitle

%\begin{center}
  %\small   Department of Mathematics,
  %  China University of Mining and Technology,
  %  Beijing 100083, P. R. China \\
%\end{center}

{\textbf{Abstract}}: Let $\mathcal{P}_r$ denote an almost--prime with at most $r$ prime factors, counted according to multiplicity. In this paper, we generalize the result of Vaughan \cite{Vaughan-1985} for ternary `admissible exponent'. Moreover, we use the refined `admissible exponent' to prove that, for $3\leqslant k\leqslant 14$ and for every sufficiently
large even integer $n$, the following equation
\begin{equation*}
    n=x^2+p_1^2+p_2^3+p_3^3+p_4^3+p_5^k
\end{equation*}
is solvable with $x$ being an almost--prime $\mathcal{P}_{r(k)}$ and the other variables primes, where $r(k)$ is defined in Theorem \ref{theorem-mixed}. This result constitutes a deepening upon that of previous results.


{\textbf{MR(2020) Subject Classification}}: 11P05, 11P32, 11P55, 11N36.

\section{Introduction and main result}

The famous Goldbach Conjecture states that every even integer $N\geqslant6$ can be written as the sum of two odd primes, i.e.
\begin{equation}\label{goldbach-con}
  N=p_1+p_2.
\end{equation}
This conjecture still remains open. The recent developments on Goldbach Conjecture can be found in
\cite{Li-Hongze-2000-1,Li-Hongze-2000-2,Lu-Wenchao-2010,Pintz-2009,Pintz-2018} and their references.

In view of Hua's theorem \cite{Hua-1938} on five squares of primes and Lagrange's theorem on four squares, it seems reasonable to conjecture that every sufficiently large integer satisfying some necessary congruence conditions can be written as the sum of four squares of primes, i.e.
\begin{equation}\label{four-squares}
  N=p_1^2+p_2^2+p_3^2+p_4^2.
\end{equation}
However, such a conjecture is out of reach at present. For the recent developments on conjecture (\ref{four-squares}), one can be found in
\cite{Harman-Kumchev-2006,Harman-Kumchev-2010,Kumchev-Zhao-2016,Liu-Jianya-Zhan-2007} and their references.

Motivated by Hua's nine cubes of primes theorem \cite{Hua-1938}, it seems reasonable to conjecture that every sufficiently large even integer is the sum of eight cubes of primes, i.e.
\begin{equation}\label{eight-cubes}
  N=p_1^3+p_2^3+p_3^3+p_4^3+p_5^3+p_6^3+p_7^3+p_8^3.
\end{equation}
But unfortunately, such a conjecture (\ref{eight-cubes}) is still out of reach at present. For the recent developments on conjecture (\ref{eight-cubes}), one can see \cite{Kawada-Zhao-2019,Kumchev-Tolev-2005} and its references.

Linnik \cite{Linnik-1959,Linnik-1960} proved that each sufficiently large odd integer $N$ can be written as $N=p+n_1^2+n_2^2$, which was firstly formulated by Hardy and Littlewood \cite{Hardy-Littlewood-1923}, where $n_1$ and $n_2$ are integers. In view of this result, it seems reasonable to conjecture that every sufficiently large integer satisfying some necessary congruence conditions is a sum of a prime and two squares of primes, i.e.
\begin{equation}\label{one-prime-two-square}
  N=p_1+p_2^2+p_3^2.
\end{equation}
But current techniques lack the power to solve it. Many authors considered this problem and gave some approaches to approximate (\ref{one-prime-two-square})
(See \cite{Harman-Kumchev-2010,Hua-1938,Leung-Liu-1993,Li-Hongze-2007,Li-Hongze-2008,Liu-Tao-2004,Lv-Sun-2009,Wang-2004,Wang-Meng-2006,Zhao-2014-JNT}). Meanwhile, we can regard this problem as the hybrid problem of (\ref{goldbach-con}) and (\ref{four-squares}).

In \cite{Liu-Zhixin-2012}, Liu considered the hybrid problem of (\ref{goldbach-con}) and (\ref{eight-cubes}), i.e.
\begin{equation}\label{one-prime-four-cubes}
  N=p_1+p_2^3+p_3^3+p_4^3+p_5^3.
\end{equation}
There are some approximations to (\ref{one-prime-four-cubes}). On one hand, as an approach to prove (\ref{one-prime-four-cubes}), Liu and L\"{u} \cite{Liu-Lv-2011} proved that every sufficiently large odd integer can be written as the sum of a prime, four cubes of primes and bounded
number of powers of 2, i.e.
\begin{equation*}
  N=p_1+p_2^3+p_3^3+p_4^3+p_5^3+2^{v_1}+2^{v_2}+\cdots+2^{v_{K_1}},
\end{equation*}
and gave an acceptable value of $K_1$. On the other hand, Liu \cite{Liu-Zhixin-2012} gave another approximation to (\ref{one-prime-four-cubes}). He proved that
every sufficiently large odd integer $N$ can be written in the form $N=x+p_1^3+ p_2^3+p_3^3+p_4^3$, where
$p_1,p_2,p_3,p_4$ are primes and $x$ is an almost--prime $\mathcal{P}_2$. As usual, $\mathcal{P}_r$ always denotes an almost--prime with at most $r$ prime factors, counted according to multiplicity. In \cite{Liu-Lv-2011}, Liu and L\"{u} also considered the hybrid problem of (\ref{four-squares}) and (\ref{eight-cubes}),
\begin{equation}\label{two-primsquare-four-cubes}
  N=p_1^2+p_2^2+p_3^3+p_4^3+p_5^3+p_6^3.
\end{equation}
In their paper, they gave an approximation to (\ref{two-primsquare-four-cubes}) and proved that every sufficiently large even integer can be written as
the sum of two squares of primes, four cubes of primes and $211$ powers of $2$, i.e.
\begin{equation}\label{Liu-res-1}
  N=p_1^2+p_2^2+p_3^3+p_4^3+p_5^3+p_6^3+2^{v_1}+2^{v_2}+\cdots+2^{v_{211}}.
\end{equation}
Later, in 2017, Liu \cite{Liu-Zhixin-2017} proved that every sufficiently large even integer can be written as the sum of two squares of primes, three cubes of primes, one fourth power of prime and a bounded number of powers of $2$, i.e.
\begin{equation}\label{Liu-res-2}
  N=p_1^2+p_2^2+p_3^3+p_4^3+p_5^3+p_6^4+2^{v_1}+2^{v_2}+\cdots+2^{v_{K_2}}.
\end{equation}
Also, in 2016, Cai \cite{Cai-2016} gave another approximation to (\ref{two-primsquare-four-cubes}), and proved that any sufficiently large even integer $N$ can be written in the form $N=x^2+p_1^2+ p_2^3+p_3^3+p_4^3+p_5^3$, where
$p_1,p_2,p_3,p_4,p_5$ are primes and $x$ is an almost--prime $\mathcal{P}_3$.

In view of the results (\ref{Liu-res-1}), (\ref{Liu-res-2}) and the result of Cai, in this paper, we shall give some approximations to the generalized cases of (\ref{two-primsquare-four-cubes}).

\begin{theorem}\label{theorem-mixed}
 For $3\leqslant k\leqslant14$, let $\mathscr{R}_k(n)$ denote the number of solutions of the equation
\begin{equation}
  n=x^2+p_1^2+p_2^3+p_3^3+p_4^3+p_5^k
\end{equation}
 with $x$ being an almost--prime $\mathcal{P}_{r(k)}$ and the $p_j$'s primes. Then, for every sufficiently large even integer $n$, there holds
 \begin{equation*}
   \mathscr{R}_k(n)\gg n^{\frac{17}{18}+\frac{5}{6k}}\log^{-6}n,
  \end{equation*}
where
\begin{align*}
 &  r(3)=3,\,\,\,\, r(4)=4,\,\,\,\, r(5)=5,\,\,\,\,\, r(6)=5,\,\,\,\,\,\, r(7)=6,\,\,\,\,\,\,\,\, r(8)=7, \\
 &  r(9)=7,\,\, r(10)=8,\,\, r(11)=9,\,\, r(12)=10,\,\, r(13)=11,\,\, r(14)=13.
\end{align*}
\end{theorem}

We approach Theorem \ref{theorem-mixed} via the Hardy--Littlewood method, and in a certain sense by a unified approach. To be specific, we use the ideas, which were firstly created by Br\"{u}dern \cite{Brudern-1995-1,Brudern-1995-2} and developed by Br\"{u}dern and Kawada \cite{Brudern-Kawada-2002,Brudern-Kawada-2009}, combining with Hardy--Littlewood method and Iwaniec's linear sieve method to give the proof of Theorem \ref{theorem-mixed}. To treat the minor arcs in the final application
of the circle method it is necessary to improve `admissible exponents' (for the definition see Section \ref{sec-AE}) for mixed sums of cubes and $k$--th powers. In the proof of Theorem \ref{theorem-mixed} we require a result on two cubes and a $k$--th power. The main idea is to apply the Hardy--Littlewood method as modified by Vaughan \cite{Vaughan-1985} to the mixed situation for one cubes and two $k$--th powers and then to combine this with the result of Vaughan \cite{Vaughan-1985}, by the Cauchy's inequality. This auxiliary result constitutes the most novel part of the present paper which may perhaps be of interest in its own right. We formulate it precisely as Theorem \ref{AE-Thm} in the following section.
Unfortunately Vaughan's elegant argument in \cite{Vaughan-1985} does not carry over very well to mixed problems; a considerable
refinement of his method will be necessary. A detailed explanation is given during the proof in Section \ref{sec-AE}.

\bigskip

\noindent
\textbf{Notation.} Throughout this paper, small italics denote integers when they do not obviously represent a
function; $p,\,p_1,\,p_2\cdots$, with or without subscript, always stand for a prime number; $\varepsilon$ always denotes an arbitrary small positive constant, which may not be the same at different occurrences; $\gamma$ denotes Euler's constant; $f(x)\ll g(x)$ means that $f(x)=O(g(x))$; $f(x)\asymp g(x)$ means that $f(x)\ll g(x)\ll f(x)$; the constants in the $O$--term and $\ll$--symbol depend at most on $\varepsilon$; $\mathcal{P}_r$ always denotes an almost--prime with at most $r$ prime factors, counted according to multiplicity. As usual, $\varphi(n),\,\mu(n)$ and $\tau_j(n)$ denote
Euler's function, M\"{o}bius' function and the $j$--dimensional divisor function respectively. Especially, we
write $\tau(n)=\tau_2(n)$.We denote by $a(m)$ and $b(\ell)$ arithmetical functions satisfying $|a(m)|\ll1$ and $|b(\ell)|\ll1$;
$(s,t)$ denotes the greatest common divisor of $s$ and $t$, while $(k;\lambda)$ is a pair of admissible exponents (see the next section); $e(\alpha)=e^{2\pi i\alpha}$ for abbreviation.

\section{Admissible Exponents for Cubes and Higher Powers}\label{sec-AE}

The idea of admissible exponents goes back to Hardy and Littlewood \cite{Hardy-Littlewood-1925}, but was introduced formally by Davenport and Erd\"{o}s \cite{Davenport-Erdos}. Our definition is adapted from Thanigasalam \cite{Thanigasalam-1980}. let
\begin{equation*}
   f_k(\alpha,X)=\sum_{X<x\leqslant 2X}e(\alpha x^k).
\end{equation*}
Let $k_i\in\mathbb{N},\,0<\lambda_i\leqslant1\,(i=1,2,\dots,s)$ and $P_i=N^{\lambda_i/k_i}$. Then the pairs $$(k_1;\lambda_1),(k_2;\lambda_2),\dots,(k_s;\lambda_s)$$
are said to form admissible exponents if
\begin{equation}\label{AE-def}
   \int_0^1 \big|f_{k_1}(\alpha,P_1)\cdots f_{k_s}(\alpha,P_s)\big|^2\mathrm{d}\alpha\ll P_1P_2\cdots P_sN^\varepsilon.
\end{equation}
This is equivalent to Thanigasalam's definition, for the integral in (\ref{AE-def}) is equal to the number of solutions of
\begin{equation*}
   x_1^{k_1}+x_2^{k_2}+\cdots+x_s^{k_s}=y_1^{k_1}+y_2^{k_2}+\cdots+y_s^{k_s}; \qquad P_i<x_i,y_i\leqslant2P_i.
\end{equation*}
Our aim is to generalize the result of Vaughan \cite{Vaughan-1985} and establish the following Theorem.

\begin{theorem}\label{AE-Thm}
 For $k\geqslant4$, the pairs $(3;1),(k;\frac{5}{6}),(k;\frac{5}{6})$ form admissible exponents.
\end{theorem}

\textbf{\textit{Proof of Theorem \ref{AE-Thm}.}} Let $Q=P^{\frac{5}{2k}}$ and let $S$ denote the number of solutions of
\begin{equation}\label{eq-zong}
   x_1^3+y_1^k+y_2^k=x_2^3+y_3^k+y_4^k
\end{equation}
with $P<x_i\leqslant2P$ and $Q<y_i\leqslant2Q$. Then we have to show that
\begin{equation*}
   S\ll P^{1+\varepsilon}Q^2.
\end{equation*}
Let $S_1$ and $S_2$ denote the number of solutions of (\ref{eq-zong}) with $x_1=x_2$ and $x_1\not=x_2$, respectively. Then, by
Hua's inequality (see Lemma 2.5 of Vaughan \cite{Vaughan-book}), it is easy to see that
\begin{equation}\label{S_1-upper}
   S_1\ll PQ^{2+\varepsilon},
\end{equation}
which is acceptable. It remains to estimate $S_2$. Write $x_2=x_1+h$. Then (\ref{eq-zong}) becomes
\begin{equation}\label{eq-zong-trans}
   h\big(3x_1^2+3x_1h+h^2\big)=y_1^k+y_2^k-y_3^k-y_4^k.
\end{equation}
By symmetry it is sufficient to estimate the solutions of (\ref{eq-zong-trans}) with $h>0$. Since $y_1^k+y_2^k\leqslant2^{k+1}Q^k$ and $x_1^2>P^2$, it follows that
\begin{equation*}
   h<\frac{2^{k+1}}{3}Q^kP^{-2}<2^kQ^kP^{-2}=2^kP^{\frac{1}{2}}=H,
\end{equation*}
say. Let
\begin{equation*}
   G(\alpha)=\sum_{0<h<H}\sum_{P<x\leqslant2P}e\Big(\alpha h\big(3x^2+3xh+h^2\big)\Big),
\end{equation*}
then
\begin{equation}\label{S_2-upp-1}
   S_2\ll \int_0^1G(\alpha)\big|f(\alpha)\big|^4\mathrm{d}\alpha
   =\int_{\frac{1}{PH}}^{1+\frac{1}{PH}}G(\alpha)\big|f(\alpha)\big|^4\mathrm{d}\alpha,
\end{equation}
where $f(\alpha)=f_{k}(\alpha,Q)$ for abbreviation. By Dirichlet's theorem on Diophantine rational approximation (for instance, see Lemma 2.1 of Vaughan \cite{Vaughan-book}), each $\alpha\in[1/(PH),1+1/(PH)]$ can be written in the form
\begin{equation*}
   \alpha=\frac{a}{q}+\lambda,\qquad |\lambda|\leqslant\frac{1}{qPH}
\end{equation*}
for some integers $a,q$ with $1\leqslant a\leqslant q\leqslant PH$ and $(a,q)=1$. Then we define the major arcs $\mathfrak{M}$ and minor arcs $\mathfrak{m}$ as follows:
\begin{equation}\label{M-m-def}
   \mathfrak{M}=\bigcup_{1\leqslant q\leqslant P}\bigcup_{\substack{1\leqslant a\leqslant q\\ (a,q)=1}}\mathfrak{M}(q,a),
   \qquad \mathfrak{m}=\bigg[\frac{1}{PH},1+\frac{1}{PH}\bigg] \setminus\mathfrak{M},
\end{equation}
where
\begin{equation*}
 \mathfrak{M}(q,a)=\bigg[\frac{a}{q}-\frac{1}{qPH},\frac{a}{q}+\frac{1}{qPH}\bigg].
\end{equation*}
Then we have
\begin{equation}\label{S_2-inte-fenjie}
 \int_{\frac{1}{PH}}^{1+\frac{1}{PH}}G(\alpha)\big|f(\alpha)\big|^4\mathrm{d}\alpha
 =\bigg\{\int_{\mathfrak{M}}+\int_{\mathfrak{m}}\bigg\}G(\alpha)\big|f(\alpha)\big|^4\mathrm{d}\alpha.
\end{equation}
According to the Lemma on p. 18 of Vaughan \cite{Vaughan-1985}, we know that
\begin{equation*}
 G(\alpha)\ll HP^{1+\varepsilon}\big(q^{-1}+P^{-1}+qP^{-2}H^{-1}\big)^{\frac{1}{2}}.
\end{equation*}
As the structure of $\mathfrak{m}$, we know that, for $\alpha\in\mathfrak{m}$, there holds $P<q\leqslant PH$, and thus
\begin{equation*}
 G(\alpha)\ll P^{1+\varepsilon},
\end{equation*}
from which and a simple consequence of Hua's lemma (Lemma 2.5 of Vaughan \cite{Vaughan-book})
\begin{equation}\label{Hua-lemma-4th}
 \int_{0}^{1}\big|f(\alpha)\big|^4\mathrm{d}\alpha\ll Q^{2+\varepsilon},
\end{equation}
we derive that
\begin{equation}\label{S_2-minor}
 \int_{\mathfrak{m}}G(\alpha)\big|f(\alpha)\big|^4\mathrm{d}\alpha\ll P^{1+\varepsilon}Q^2.
\end{equation}
From (\ref{S_2-upp-1}), (\ref{S_2-inte-fenjie}) and (\ref{S_2-minor}), we deduce that
\begin{equation}\label{S_2-upper-2}
 S_2\ll \int_{\mathfrak{M}}G(\alpha)\big|f(\alpha)\big|^4\mathrm{d}\alpha+ P^{1+\varepsilon}Q^2.
\end{equation}
In order to estimate the integral on the major arcs, we approximate $G(\alpha)$ by a suitable function $G_1(\alpha)$. Define
\begin{align*}
   & \sigma_h(q,a)=\sum_{x=1}^qe\bigg(\frac{a}{q}\big((x+h)^3-x^3\big)\bigg),    \\
   &    v_h(\lambda)=\int_P^{2P}e\Big(\lambda\big((u+h)^3-u^3\big)\Big)\mathrm{d}u,   \\
   &G_1(\alpha)=\sum_{0<h<H}q^{-1}\sigma_h(q,a)v_h\bigg(\alpha-\frac{a}{q}\bigg).
\end{align*}
Then for $\alpha\in\mathfrak{M}(q,a)$, $G_1(\alpha)$ is well defined on $\mathfrak{M}$. By (2.13) of Lemma 2 in Vaughan
\cite{Vaughan-1986} with $k=3$, one has
\begin{equation*}
 \sum_{P<x\leqslant 2P}e\Big(\alpha h(3x^2+3xh+h^2)\Big)=q^{-1}\sigma_h(q,a)v_h(\lambda)
 +O\Big(q^{\frac{1}{2}+\varepsilon}(q,h)^{\frac{1}{2}}\Big),
\end{equation*}
from which we obtain
\begin{equation}\label{G=G_1+E}
G(\alpha)=G_1(\alpha)+O\Bigg(q^{\frac{1}{2}+\varepsilon}\sum_{0<h<H}(q,h)^{\frac{1}{2}}\Bigg).
\end{equation}
For the $O$--term in (\ref{G=G_1+E}), writing $(q,h)=d$, we see that
\begin{align*}
              q^{\frac{1}{2}+\varepsilon}\sum_{0<h<H}(q,h)^{\frac{1}{2}}
  \ll  &\,\,  q^{\frac{1}{2}+\varepsilon}\sum_{d|q}\sum_{h<H/d}d^{\frac{1}{2}} \ll q^{\frac{1}{2}+\varepsilon}
              \sum_{d|q}d^{\frac{1}{2}}\cdot\frac{H}{d}
                       \nonumber \\
  \ll  &\,\   Hq^{\frac{1}{2}+\varepsilon}\tau(q)\ll Hq^{\frac{1}{2}+\varepsilon},
\end{align*}
from which and (\ref{G=G_1+E}) we derive that
\begin{equation}\label{G=G_1+E-2}
G(\alpha)=G_1(\alpha)+O\big(P^{1+\varepsilon}\big)
\end{equation}
uniformly for $\alpha\in\mathfrak{M}$. Combining (\ref{Hua-lemma-4th}), (\ref{S_2-upper-2}) and (\ref{G=G_1+E-2}), we have
\begin{equation}\label{S_2-upper-3}
 S_2\ll \int_{\mathfrak{M}}G_1(\alpha)\big|f(\alpha)\big|^4\mathrm{d}\alpha+ P^{1+\varepsilon}Q^2.
\end{equation}
In order to give a proper upper bound for the integral on the right--hand side of (\ref{S_2-upper-3}), we need to establish
the following lemma, which is the crucial ingredient of this section.
\begin{lemma}\label{AE-lemma}
Let $\mathfrak{M}$ be defined as in (\ref{M-m-def}),then for $k\geqslant4$ and $X\leqslant P$, there holds
\begin{equation*}
     \int_{\mathfrak{M}}\big|G_1^2(\alpha)f_k^4(\alpha,X)\big|\mathrm{d}\alpha\ll HP^\varepsilon\big(PX^2+X^4\big).
\end{equation*}
\end{lemma}
First of all, we use Lemma \ref{AE-lemma} to give the expected estimate of the integral on the right--hand side of (\ref{S_2-upper-3}) and prove it afterwards. Taking $X=Q$ and $f(\alpha)=f_k(\alpha,Q)$ in Lemma \ref{AE-lemma}, then it follows from (\ref{Hua-lemma-4th}), Lemma \ref{AE-lemma} and Cauchy's inequality that
\begin{align}\label{yong-lemma}
             \int_{\mathfrak{M}}G_1(\alpha)\big|f(\alpha)\big|^4\mathrm{d}\alpha
  \ll & \,\, \bigg(\int_{\mathfrak{M}}\big|G_1^2(\alpha)f^4(\alpha)\big|\mathrm{d}\alpha \bigg)^{\frac{1}{2}}
             \bigg(\int_0^1\big|f(\alpha)\big|^4\mathrm{d}\alpha\bigg)^{\frac{1}{2}}
                        \nonumber \\
  \ll & \,\, \Big(HP^\varepsilon\big(PQ^2+Q^4\big)\Big)^{\frac{1}{2}}\big(Q^{2+\varepsilon}\big)^{\frac{1}{2}}
             \ll H^{\frac{1}{2}}P^\varepsilon \big(P^{\frac{1}{2}}Q+Q^2\big)Q
                        \nonumber \\
  \ll & \,\, P^\varepsilon\big(H^{\frac{1}{2}}P^{\frac{1}{2}}Q^2+H^{\frac{1}{2}}Q^3\big)\ll P^{1+\varepsilon}Q^2.
\end{align}
From (\ref{S_1-upper}), (\ref{S_2-upper-3}) and (\ref{yong-lemma}), we derive the conclusion of Theorem \ref{AE-Thm}.

\textbf{\textit{Proof of Lemma \ref{AE-lemma}.}} By Theorem 7.1 of Vaughan \cite{Vaughan-book}, it is easy to see that
\begin{equation*}
  \sigma_h(q,a)\ll q^{\frac{1}{2}+\varepsilon}(q,h)^{\frac{1}{2}}.
\end{equation*}
For $\alpha\in\mathfrak{M}(q,a)$, it follows from Cauchy's inequality that
\begin{align}\label{G_1-upp-1}
          \big|G_1(\alpha)\big|^2
 = & \,\, \Bigg|\sum_{0<h<H}q^{-1}\sigma_h(q,a)v_h(\lambda)\Bigg|^2
          \ll P^\varepsilon \Bigg|\sum_{0<h<H}\frac{(q,h)^{1/2}}{q^{1/2}}\big|v_h(\lambda)\big|\Bigg|^2
                    \nonumber \\
 \ll & \,\, P^\varepsilon\Bigg(\sum_{0<h<H}\frac{1}{q}\Bigg)\Bigg(\sum_{0<h<H}(q,h)\big|v_h(\lambda)\big|^2\Bigg)
                     \nonumber \\
 \ll & \,\, P^\varepsilon Hq^{-1}\sum_{0<h<H}(q,h)\big|v_h(\lambda)\big|^2.
\end{align}
By the standard estimate
\begin{equation*}
  v_h(\lambda)\ll \frac{P}{1+P^2h|\lambda|},
\end{equation*}
which combines (\ref{G_1-upp-1}) to give
\begin{align*}
             \int_{\mathfrak{M}}\big|G_1^2(\alpha)f_k^4(\alpha,X)\big|\mathrm{d}\alpha
 \ll  & \,\, P^\varepsilon H\sum_{0<h<H}\int_{\mathfrak{M}}\frac{(q,h)}{q}\big|v_h(\lambda)\big|^2
             \big|f_k^4(\alpha,X)\big|\mathrm{d}\alpha
                  \nonumber \\
 \ll & \,\, P^\varepsilon H\sum_{0<h<H}\sum_{1\leqslant q\leqslant P}\sum_{\substack{a=1\\ (a,q)=1}}^q\frac{(q,h)}{q}
             \int_{-\frac{1}{qPH}}^{\frac{1}{qPH}}\big|v_h(\lambda)\big|^2
                  \nonumber \\
   & \,\, \times \sum_{X<x_1,\dots,x_4\leqslant2X}
             e\Bigg(\bigg(\frac{a}{q}+\lambda\bigg)\big(x_1^k+x_2^k-x_3^k-x_4^k\big)\Bigg)\mathrm{d}\lambda.
\end{align*}
Setting $u=x_1^k+x_2^k-x_3^k-x_4^k$, then
\begin{align}\label{S_2-upper-4}
     & \,\, \int_{\mathfrak{M}}\big|G_1^2(\alpha)f_k^4(\alpha,X)\big|\mathrm{d}\alpha
                    \nonumber \\
 \ll & \,\, P^\varepsilon H\sum_{0<h<H}\sum_{1\leqslant q\leqslant P}\sum_{\substack{a=1\\ (a,q)=1}}^q\frac{(q,h)}{q}
            \sum_{u}\varrho(u)\int_{-\frac{1}{qPH}}^{\frac{1}{qPH}}\big|v_h(\lambda)\big|^2
            e\bigg(\bigg(\frac{a}{q}+\lambda\bigg)u\bigg)\mathrm{d}\lambda
                  \nonumber \\
  \ll & \,\, P^\varepsilon H\sum_{0<h<H}\sum_{1\leqslant q\leqslant P}\frac{(q,h)}{q}\sum_{u}\varrho(u)
             \Bigg|\sum_{\substack{a=1\\ (a,q)=1}}^qe\bigg(\frac{au}{q}\bigg)\Bigg|
             \int_{-\frac{1}{qPH}}^{\frac{1}{qPH}}\big|v_h(\lambda)\big|^2\mathrm{d}\lambda
                  \nonumber \\
  \ll & \,\, P^\varepsilon H\sum_{0<h<H}\sum_{1\leqslant q\leqslant P}\frac{(q,h)}{q}\sum_{u}\varrho(u)
             \Bigg|\sum_{\substack{a=1\\ (a,q)=1}}^qe\bigg(\frac{au}{q}\bigg)\Bigg|
             \int_{|\lambda|\leqslant\frac{1}{PH}}\frac{P^2}{(1+P^2h|\lambda|)^2}\mathrm{d}\lambda
                   \nonumber \\
  \ll & \,\, P^\varepsilon H\sum_{0<h<H}\sum_{1\leqslant q\leqslant P}\frac{(q,h)}{q}\sum_{u}\varrho(u)
             \Bigg|\sum_{\substack{a=1\\ (a,q)=1}}^qe\bigg(\frac{au}{q}\bigg)\Bigg|
                   \nonumber \\
    & \,\,   \qquad \qquad \times\Bigg(\int_{|\lambda|\leqslant\frac{1}{P^2h}}P^2\mathrm{d}\lambda+
             \int_{\frac{1}{P^2h}<|\lambda|\leqslant\frac{1}{PH}}\frac{P^2}{P^4h^2|\lambda|^2}\mathrm{d}\lambda\Bigg)
                    \nonumber \\
  \ll & \,\, P^\varepsilon H\sum_{0<h<H}\sum_{1\leqslant q\leqslant P}\frac{(q,h)}{qh}\sum_{u}\varrho(u)
             \Bigg|\sum_{\substack{a=1\\ (a,q)=1}}^qe\bigg(\frac{au}{q}\bigg)\Bigg|,
\end{align}
where $\varrho(u)$ denotes the number of solutions of $u=x_1^k+x_2^k-x_3^k-x_4^k$ with $X<x_i\leqslant2X\,(i=1,2,3,4)$. By
Hua's lemma (Lemma 2.5 of Vaughan \cite{Vaughan-book}), we have $\varrho(0)\ll X^{2+\varepsilon}$. For $u\not=0$, it follows from Theorem 271 of Hardy and Wright \cite{Hardy-Wright-1979} that
\begin{equation*}
  \sum_{\substack{a=1\\(a,q)=1}}^qe\bigg(\frac{au}{q}\bigg)=\sum_{d|(q,u)}\mu\bigg(\frac{q}{d}\bigg)d.
\end{equation*}
Thus, the right--hand side of (\ref{S_2-upper-4}) is bounded by
\begin{align}\label{S_2-upper-5}
  \ll & \,\, P^\varepsilon H\Bigg(X^2\sum_{0<h<H}\sum_{1\leqslant q\leqslant P}\frac{(q,h)}{h}
             +\sum_{0<h<H}\sum_{1\leqslant q\leqslant P}\frac{(q,h)}{qh}\sum_{u\not=0}\varrho(u)\sum_{d|(q,u)}d\Bigg)
                            \nonumber \\
  \ll & \,\, P^\varepsilon H\big(\Sigma_1+\Sigma_2\big),
\end{align}
say. Writing $r=(q,h)$, then $q=rq_1,h=rh_1$ with $(q_1,h_1)=1$. Thus, we have
\begin{align}\label{Sigma_1-upper}
            \Sigma_1
 \ll & \,\, X^2\sum_{1\leqslant r\leqslant P}\sum_{1\leqslant h_1<H/r}\sum_{1\leqslant q_1\leqslant P/r}\frac{1}{h_1}
                           \nonumber \\
 \ll & \,\, X^2\Bigg(\sum_{1\leqslant r\leqslant P}\frac{P}{r}\Bigg)\Bigg(\sum_{1\leqslant h_1\leqslant H/r}\frac{1}{h_1}\Bigg)
                           \nonumber \\
 \ll & \,\, X^2P\Bigg(\sum_{1\leqslant r\leqslant P}\frac{\log(H/r)}{r}\Bigg)\ll X^2P^{1+\varepsilon}.
\end{align}
For $\Sigma_2$, by the same transformation, we obtain
\begin{equation*}
  \Sigma_2\ll \sum_{1\leqslant r\leqslant P}\sum_{1\leqslant h_1<H/r}\sum_{1\leqslant q_1\leqslant P/r}\frac{1}{q_1h_1r}
              \sum_{u\not=0} \varrho(u)\sum_{d|(rq_1,u)}d.
\end{equation*}
We first consider the inner double sums over $u$ and $d$, and see that $d|(rq_1,u)$ implies $d|u$ and $rq_1=ds$ for some integer $s$. Moreover, for fixed $d$ and $s$, there exist $O(P^\varepsilon)$ solutions of $rq_1=ds$ in integer variables $r$
and $q_1$. Hence, we deduce that
\begin{align}\label{Sigma_2-upper}
            \Sigma_2
 \ll & \,\, P^\varepsilon\sum_{u\not=0}\varrho(u)\sum_{d|u}d\sum_{s\leqslant P}\sum_{h_1<H}\frac{1}{h_1ds}
                 \nonumber \\
 \ll & \,\, P^\varepsilon\sum_{u\not=0}\varrho(u)\Bigg(\sum_{d|u}1\Bigg)\Bigg(\sum_{s\leqslant P}\frac{1}{s}\Bigg)
            \Bigg(\sum_{h_1<H}\frac{1}{h_1}\Bigg)
                 \nonumber \\
 \ll & \,\, P^\varepsilon\sum_{u\not=0}\varrho(u)\ll X^4P^\varepsilon.
\end{align}
Combining (\ref{S_2-upper-4}), (\ref{S_2-upper-5}), (\ref{Sigma_1-upper}) and (\ref{Sigma_2-upper}), we get the conclusion of Lemma \ref{AE-lemma}.

From Theorem \ref{AE-Thm} and the Theorem of Vaughan \cite{Vaughan-1985}, we obtain the following corollary.

\begin{corollary}\label{AE-corollary}
For  $k\geqslant3$, the pairs $(3;1),(3;\frac{5}{6}),(k;\frac{5}{6})$ form admissible exponents.
\end{corollary}
\textbf{\textit{Proof of corollary \ref{AE-corollary}.}} For $k=3$, the conclusion follows from the Theorem of Vaughan \cite{Vaughan-1985}. For $k\geqslant4$, by the Theorem of Vaughan \cite{Vaughan-1985}, Theorem \ref{AE-Thm} and Cauchy's
inequality, we deduce that
\begin{align*}
   & \,\,  \int_0^1\Big|f_3\big(\alpha,N^{\frac{1}{3}}\big)f_3\big(\alpha,N^{\frac{5}{18}}\big)
           f_k\big(\alpha,N^{\frac{5}{6k}}\big)\Big|^2\mathrm{d}\alpha
                 \nonumber \\
\ll & \,\,  \Bigg(\int_0^1\Big|f_3\big(\alpha,N^{\frac{1}{3}}\big)f_3^2\big(\alpha,N^{\frac{5}{18}}\big)\Big|^2
            \mathrm{d}\alpha\Bigg)^\frac{1}{2}
            \Bigg(\int_0^1\Big|f_3\big(\alpha,N^{\frac{1}{3}}\big)f_k^2\big(\alpha,N^{\frac{5}{6k}}\big)\Big|^2
            \mathrm{d}\alpha\Bigg)^\frac{1}{2}
                   \nonumber \\
\ll & \,\,  \big(N^{\frac{1}{3}+\frac{5}{18}+\frac{5}{18}+\varepsilon}\big)^{\frac{1}{2}}
            \big(N^{\frac{1}{3}+\frac{5}{6k}+\frac{5}{6k}+\varepsilon}\big)^{\frac{1}{2}}
             \ll N^{\frac{1}{3}+\frac{5}{18}+\frac{5}{6k}+\varepsilon},
\end{align*}
which implies $(3;1),(3;\frac{5}{6}),(k;\frac{5}{6})$ form admissible exponents for $k\geqslant4$.

\section{Proof of Theorem \ref{theorem-mixed}: Preliminaries}
In this section, we shall give some notations and preliminary lemmas. We always denote by $\chi$ a Dirichlet character $\pmod q$, and by $\chi^0$ the principal Dirichlet character$\pmod q$. Let
\begin{equation*}
A=10^{200}, \qquad  Q_0=\log^{50A}n, \qquad Q_1=n^{\frac{5}{9}-\frac{5}{6k}+50\varepsilon}, \qquad Q_2=n^{\frac{4}{9}+\frac{5}{6k}-50\varepsilon},
\end{equation*}
\begin{equation*}
D=n^{\frac{5}{8k}-\frac{1}{24}-51\varepsilon},\quad z=D^{\frac{1}{3}},\quad X_j=\frac{1}{2}\bigg(\frac{2n}{3}\bigg)^{\frac{1}{j}},\quad
X_j^*=\frac{1}{2}\bigg(\frac{2n}{3}\bigg)^{\frac{5}{6j}},\quad \mathscr{P}=\prod_{2<p<z}p,
\end{equation*}
\begin{equation*}
F_j(\alpha)=\sum_{X_j<m\leqslant 2X_j}e(m^j\alpha),\,\,\,\,
f_j(\alpha)=\sum_{X_j<p\leqslant2X_j}(\log p)e(p^j\alpha), \,\,\,\,  w_j(\lambda)=\int_{X_j}^{2X_j}e(\lambda u^j)\mathrm{d}u,
\end{equation*}
\begin{equation*}
F_j^*(\alpha)=\sum_{X_j^*<m\leqslant 2X_j^*}e(m^j\alpha),\,\,
f_j^*(\alpha)=\sum_{X_j^*<p\leqslant 2X_j^*}(\log p)e(p^j\alpha), \,\,
w_j^*(\lambda)=\int_{X_j^*}^{2X_j^*}e(\lambda u^j)\mathrm{d}u,
\end{equation*}
\begin{equation*}
 G_j(\chi,a)=\sum_{m=1}^q\chi(m)e\bigg(\frac{am^j}{q}\bigg),\quad S^*_j(q,a)=G_j(\chi^0,a), \quad S_j(q,a)=\sum_{m=1}^qe\bigg(\frac{am^j}{q}\bigg),
\end{equation*}
\begin{equation*}
h(\alpha)=\sum_{m\leqslant D^{2/3}}a(m)\sum_{s\leqslant D^{1/3}}
b(s)\sum_{\frac{X_2}{ms}<t\leqslant\frac{2X_2}{ms}}e\big((mst)^2\alpha\big),
\end{equation*}
\begin{equation*}
  B_d(q,n)=\sum_{\substack{a=1\\(a,q)=1}}^q S_2(q,ad^2)S_2^*(q,a)S_3^{*3}(q,a)S_k^*(q,a)e\bigg(-\frac{an}{q}\bigg),
\end{equation*}
\begin{equation*}
  B(q,n)=B_1(q,n),\qquad  A_d(q,n)=\frac{B_d(q,n)}{q\varphi^5(q)},\qquad  A(q,n)=A_1(q,n),
\end{equation*}
\begin{equation*}
  \mathfrak{S}_d(n)=\sum_{q=1}^\infty A_d(q,n),\qquad\quad \mathfrak{S}(n)=\mathfrak{S}_1(n),
\end{equation*}
\begin{equation*}
\mathcal{J}(n)=\int_{-\infty}^{+\infty}w_2^2(\lambda)w_3^2(\lambda)w_3^*(\lambda)w_k^*(\lambda)e(-n\lambda)\mathrm{d}\lambda,
\end{equation*}
\begin{equation*}
\mathcal{B}_r=\big\{m:\,X_2<m\leqslant2X_2,\,\,m=p_1p_2\cdots p_r,\,\, z\leqslant p_1\leqslant p_2\leqslant\cdots\leqslant p_r\big\},
\end{equation*}
\begin{equation*}
\mathcal{N}_r=\big\{m:\,\,m=p_1p_2\cdots p_{r-1},\,\, z\leqslant p_1\leqslant p_2\leqslant\cdots\leqslant p_{r-1},\,
p_1p_2\cdots p_{r-2}p_{r-1}^2\leqslant 2X_2\big\},
\end{equation*}
\begin{equation*}
g_r(\alpha)=\sum_{\substack{X_2<\ell p\leqslant2X_2\\ \ell\in\mathcal{N}_r}}\frac{\log p}{\log\frac{X_2}{\ell}}
e\Big(\alpha\big(\ell p\big)^2\Big),
 \quad  \log{\bf{\Xi}}=(\log2X_2)(\log2X_3)^2(\log2X_3^*)(\log2X_k^*),
\end{equation*}
\begin{equation*}
\log{\bf{\Theta}}=(\log X_3)^2(\log X_3^*)(\log X_k^*).
\end{equation*}

\begin{lemma}\label{Hua-compo}
For $(a,q)=1$, we have
\begin{eqnarray*}
  &  \emph{(i)} & S_{j}(q,a)\ll q^{1-\frac{1}{j}};   \\
  & \emph{(ii)} & G_{j}(\chi,a)\ll q^{\frac{1}{2}+\varepsilon}.
\end{eqnarray*}
In particular, for $(a,p)=1$, we have
\begin{eqnarray*}
  &  \emph{(iii)} &   |S_{j}(p,a)|\leqslant \big((j,p-1)-1\big)\sqrt{p};    \\
  & \emph{(iv)}   &   |S_{j}^*(p,a)|\leqslant \big((j,p-1)-1\big)\sqrt{p}+1;  \\
  & \emph{(v)}    &   S_{j}^*(p^\ell,a)=0\,\,\, \textrm{for}\,\,\, \ell\geqslant\gamma(p),\,\,\textrm{where}  \\
  &  &  \,\,\gamma(p)=
                   \begin{cases}
                        \theta+2, & \textrm{if}\,\,\, p^\theta\|j,\,\, p\not=2\,\,\,\textrm{or}\,\,\, p=2,\,\theta=0,   \\
                        \theta+3, & \textrm{if}\,\,\, p^{\theta}\|j,\,\,p=2,\,\,\theta>0.
                   \end{cases}
\end{eqnarray*}
\end{lemma}
\begin{proof}
  For (i) and (iii)--(iv), see Theorem 4.2 and Lemma 4.3 of Vaughan \cite{Vaughan-book}, respectively. For (ii), see Lemma 8.5 of Hua \cite{Hua-book} or
  the Problem 14 of Chapter VI of Vinogradov \cite{Vinogradov-book}. For (v), see Lemma 8.3 of Hua \cite{Hua-book}.
\end{proof}

\begin{lemma}\label{mean-value}
We have
\begin{align*}
  \textrm{\emph{(i)}} & \,\,\int_0^1\big|F_2(\alpha)F_3(\alpha)F_3^*(\alpha)F_k^*(\alpha)\big|^2\mathrm{d}\alpha
                       \ll n^{\frac{11}{9}+\frac{5}{3k}}(\log n)^c,
                           \nonumber \\
   \textrm{\emph{(ii)}} & \,\,\int_0^1\big|f_2(\alpha)f_3(\alpha)f_3^*(\alpha)f_k^*(\alpha)\big|^2\mathrm{d}\alpha
                       \ll n^{\frac{11}{9}+\frac{5}{3k}}(\log n)^{c+8},
\end{align*}
where $c$ is an absolute constant.
\end{lemma}
\begin{proof}
   By Lemma 2.4 of Cai \cite{Cai-2014}, we know that
\begin{equation*}
\int_0^1\big|F_3(\alpha)F_3^*(\alpha)\big|^4\mathrm{d}\alpha\ll n^{\frac{13}{9}}.
\end{equation*}
By the above estimate, Cauchy's inequality and the explicit form of Hua's inequality (see Theorem 4 on p. 19 of Hua \cite{Hua-book}), we deduce that
\begin{align}\label{mean-le}
     & \,\, \int_0^1\big|F_2(\alpha)F_3(\alpha)F_3^*(\alpha)\big|^2\mathrm{d}\alpha
                \nonumber \\
\ll & \,\,  \Bigg(\int_0^1\big|F_2(\alpha)\big|^4\mathrm{d}\alpha\Bigg)^{\frac{1}{2}}
            \Bigg(\int_0^1\big|F_3(\alpha)F_3^*(\alpha)\big|^4\mathrm{d}\alpha\Bigg)^{\frac{1}{2}}
                \nonumber \\
\ll & \,\,  \big(X_2^2(\log X_2)^c\big)^{\frac{1}{2}}\big(n^{\frac{13}{9}}\big)^{\frac{1}{2}}\ll n^{\frac{11}{9}}(\log n)^c.
\end{align}
Then (i) follows from (\ref{mean-le}) and the trivial estimate $|F_k^*(\alpha)|\ll n^{\frac{5}{6k}}$. Moreover, by considering
the number of solutions of the underlying Diophantine equation and the result of (i), we obtain the estimate (ii). This completes the proof of Lemma \ref{mean-value}.
\end{proof}

\begin{lemma}\label{W(a)-Delta-k}
  For $\alpha=\frac{a}{q}+\lambda$, define
\begin{equation}\label{M(q,a)-def}
   \mathcal{M}(q,a)=\bigg(\frac{a}{q}-\frac{1}{qQ_0},\frac{a}{q}+\frac{1}{qQ_0}\bigg],
\end{equation}
\begin{equation}\label{Delta_k-def}
   \Delta_k(\alpha)=f_k(\alpha)-\frac{S_k^*(q,a)}{\varphi(q)}\sum_{X_k<m\leqslant 2X_k}e(m^k\lambda),
\end{equation}
\begin{equation}\label{W(alpha)-def}
   \mathcal{W}(\alpha)=\sum_{d\leqslant D}\frac{c(d)}{dq}S_2(q,ad^2)w_2(\lambda),
\end{equation}
where
\begin{equation*}
  c(d)=\sum_{\substack{d=m\ell\\ m\leqslant D^{2/3}\\ \ell\leqslant D^{1/3}}}a(m)b(\ell)\ll \tau(d).
\end{equation*}
Then we have
\begin{equation*}
   \sum_{1\leqslant q\leqslant Q_0}
   \sum_{\substack{a=-q\\(a,q)=1}}^{2q}\int_{\mathcal{M}(q,a)}\big|\mathcal{W}(\alpha)\Delta_k(\alpha)\big|^2\mathrm{d}\alpha
   \ll n^{\frac{2}{k}}\log^{-100A}n.
\end{equation*}
\end{lemma}
\begin{proof}
 See Lemma 2.4 of Li and Cai \cite{Li-Cai-2016}.
\end{proof}

\begin{lemma}\label{two-mean}
   For $\alpha=\frac{a}{q}+\lambda$, define
   \begin{equation}\label{V_k-def}
      \mathcal{V}_k(\alpha)=\frac{S_k^*(q,a)}{\varphi(q)}w_k(\lambda).
   \end{equation}
 Then we have
 \begin{equation*}
   \sum_{1\leqslant q\leqslant Q_0}
   \sum_{\substack{a=-q\\(a,q)=1}}^{2q}\int_{\mathcal{M}(q,a)}\big|\mathcal{V}_k(\alpha)\big|^2\mathrm{d}\alpha
   \ll n^{\frac{2}{k}-1}\log^{21A}n,
\end{equation*}
and
\begin{equation*}
   \sum_{1\leqslant q\leqslant Q_0}
   \sum_{\substack{a=-q\\(a,q)=1}}^{2q}\int_{\mathcal{M}(q,a)}\big|\mathcal{W}(\alpha)\big|^2\mathrm{d}\alpha
   \ll \log^{21A}n,
\end{equation*}
where $\mathcal{M}(q,a)$ and $\mathcal{W}(\alpha)$ are defined by (\ref{M(q,a)-def}) and (\ref{W(alpha)-def}), respectively.
\end{lemma}
\begin{proof}
 See Lemma 2.5 of Li and Cai \cite{Li-Cai-2016}.
\end{proof}

For $(a,q)=1,\,1\leqslant a\leqslant q\leqslant Q_2$, define
\begin{equation*}
  \mathscr{M}(q,a)=\bigg(\frac{a}{q}-\frac{1}{qQ_2},\frac{a}{q}+\frac{1}{qQ_2}\bigg],\qquad
  \mathscr{M}=\bigcup_{1\leqslant q\leqslant Q_0^5}\bigcup_{\substack{1\leqslant a\leqslant q\\(a,q)=1}}\mathscr{M}(q,a),
\end{equation*}
\begin{equation*}
  \mathscr{M}_0(q,a)=\bigg(\frac{a}{q}-\frac{Q_0}{n},\frac{a}{q}+\frac{Q_0}{n}\bigg],\qquad
  \mathscr{M}_0=\bigcup_{1\leqslant q\leqslant Q_0^5}\bigcup_{\substack{1\leqslant a\leqslant q\\(a,q)=1}}\mathscr{M}_0(q,a),
\end{equation*}
\begin{equation*}
  \mathcal{I}_0=\bigg(-\frac{1}{Q_2},1-\frac{1}{Q_2}\bigg],\quad\quad \quad \mathfrak{m}_0=\mathscr{M}\setminus \mathscr{M}_0,
\end{equation*}
\begin{equation*}
  \mathfrak{m}_1=\bigcup_{Q_0^5<q\leqslant Q_1}\bigcup_{\substack{1\leqslant a\leqslant q\\(a,q)=1}}\mathscr{M}(q,a),\quad\quad\quad
  \mathfrak{m}_2=\mathcal{I}_0\setminus (\mathscr{M}\cup \mathfrak{m}_1).
\end{equation*}
Then we obtain the Farey dissection
\begin{equation}\label{Farey-dissection}
  \mathcal{I}_0=\mathscr{M}_0\cup\mathfrak{m}_0\cup\mathfrak{m}_1\cup\mathfrak{m}_2.
\end{equation}

\begin{lemma}\label{substi-lemma}
   For $\alpha=\frac{a}{q}+\lambda$, define
\begin{equation*}
  \mathcal{V}_k^*(\alpha)=\frac{S_k^*(q,a)}{\varphi(q)}w_k^*(\lambda).
\end{equation*}
Then for $\alpha=\frac{a}{q}+\lambda\in\mathscr{M}_0$, we have
   \begin{align*}
      \emph{(i)} &\,\,  f_j(\alpha)=\mathcal{V}_j(\alpha)+O\big(X_j\exp(-\log^{1/3}n)\big),
                \nonumber \\
     \emph{(ii)} &\,\, f_j^*(\alpha)=\mathcal{V}_j^*(\alpha)+O\big(X_j^*\exp(-\log^{1/3}n)\big),
                \nonumber \\
    \emph{(iii)} &\,\, g_r(\alpha)=\frac{c_r(k)\mathcal{V}_2(\alpha)}{\log X_2}+O\big(X_2\exp(-\log^{1/3}n)\big),
   \end{align*}
where $\mathcal{V}_j(\alpha)$ is defined (\ref{V_k-def}), and
\begin{align}\label{c_r-def}
 c_r(k)  = &(1+O(\varepsilon)) \nonumber  \\
    & \times\int_{r-1}^{\frac{37k-15}{15-k}}\frac{\mathrm{d}t_1}{t_1}\int_{r-2}^{t_1-1}\frac{\mathrm{d}t_2}{t_2}\cdots
         \int_{3}^{t_{r-4}-1}\frac{\mathrm{d}t_{r-3}}{t_{r-3}}\int_{2}^{t_{r-3}-1}\frac{\log (t_{r-2}-1)}{t_{r-2}}\mathrm{d}t_{r-2}.
\end{align}
\end{lemma}
\begin{proof}
By some routine arguments and partial summation, (i)--(iii) follow from Siegel--Walfisz theorem
and prime number theorem.
\end{proof}

\begin{lemma}\label{h(a)-upp}
  For $\alpha\in\mathfrak{m}_2$, we have
  \begin{equation*}
     h(\alpha)\ll n^{\frac{2}{9}+\frac{5}{12k}-24\varepsilon}.
  \end{equation*}
\end{lemma}
\begin{proof}
  By the estimate (4.5) of Lemma 4.2 in Br\"{u}dern and Kawada \cite{Brudern-Kawada-2002}, we deduce that
\begin{align*}
   h(\alpha) \ll & \,\, \frac{n^{\frac{1}{2}}\tau^2(q)\log^2n}{(q+n|q\alpha-a|)^{1/2}}+n^{\frac{1}{4}+\varepsilon}D^{\frac{2}{3}}  \\
   \ll & \,\, n^{\frac{1}{2}+\varepsilon}Q_1^{-\frac{1}{2}}+n^{\frac{1}{4}+\varepsilon}D^{\frac{2}{3}}\ll n^{\frac{2}{9}+\frac{5}{12k}-24\varepsilon},
\end{align*}
which completes the proof of Lemma \ref{h(a)-upp}.
\end{proof}

\section{Mean Value Theorems}

In this section, we shall prove the mean value theorems for the proof of Theorem \ref{theorem-mixed}.
\begin{proposition}\label{mean-value-1}
For $3\leqslant k\leqslant 14$, define
\begin{equation*}
  J(n,d)=\sum_{\substack{m^2+p_1^2+p_2^3+p_3^3+p_4^3+p_5^k=n\\ X_2<m,\,p_1\leqslant2X_2, \,\, m\equiv0 \!\!\!\!\!\pmod d \\
               X_3<p_2,p_3\leqslant2X_3 \\ X_3^*<p_4\leqslant2X_3^* \\ X_k^*<p_5\leqslant2X_k^*}}
               \prod_{j=1}^5\log p_j.
\end{equation*}
Then we have
\begin{equation*}
  \sum_{m\leqslant D^{2/3}}a(m)\sum_{\ell\leqslant D^{1/3}}b(\ell)
     \bigg(J(n,m\ell)-\frac{\mathfrak{S}_{m\ell}(n)}{m\ell}\mathcal{J}(n)\bigg)
     \ll n^{\frac{17}{18}+\frac{5}{6k}}\log^{-A}n.
\end{equation*}
\end{proposition}

\begin{proof}
  Let
\begin{equation*}
  \mathcal{K}(\alpha)=h(\alpha)f_2(\alpha)f_3^2(\alpha)f_3^*(\alpha)f_k^*(\alpha)e(-n\alpha).
\end{equation*}
  By the Farey dissection (\ref{Farey-dissection}), we have
\begin{align}\label{mean-fenjie}
  &\,\, \sum_{m\leqslant D^{2/3}}a(m)\sum_{\ell\leqslant D^{1/3}}b(\ell)J(n,m\ell)
              \nonumber \\
  =& \,\,\int_{\mathcal{I}_0}\mathcal{K}(\alpha)\mathrm{d}\alpha
         =\bigg(\int_{\mathscr{M}_0}+\int_{\mathfrak{m}_0}+\int_{\mathfrak{m}_1}+\int_{\mathfrak{m}_2}\bigg)
           \mathcal{K}(\alpha)\mathrm{d}\alpha.
\end{align}
From Hua's lemma (see Lemma 2.5 of Vaughan \cite{Vaughan-book}), Corollary \ref{AE-corollary}  and H\"{o}lder's inequality,
we obtain
\begin{align}\label{mean-jun}
        &  \,\,\int_0^1 \big|f_2(\alpha)f_3^2(\alpha)f_3^*(\alpha)f_k^*(\alpha)\big|\mathrm{d}\alpha
                              \nonumber  \\
    \ll &  \,\,     \bigg(\int_0^1\big|f_2(\alpha)\big|^4\mathrm{d}\alpha \bigg)^{\frac{1}{4}}
                    \bigg(\int_0^1\big|f_3(\alpha)\big|^4\mathrm{d}\alpha \bigg)^{\frac{1}{4}}
                    \bigg(\int_0^1\big|f_3(\alpha)f_3^*(\alpha)f_k^*(\alpha)\big|^2\mathrm{d}\alpha \bigg)^{\frac{1}{2}}
                              \nonumber  \\
    \ll & \,\,  \big(X_2^{2+\varepsilon}\big)^{\frac{1}{4}}\big(X_3^{2+\varepsilon}\big)^{\frac{1}{4}}
                \big(n^{\frac{1}{3}+\frac{5}{18}+\frac{5}{6k}+\varepsilon}\big)^{\frac{1}{2}}
                \ll n^{\frac{13}{18}+\frac{5}{12k}+\varepsilon}.
\end{align}
  By Lemma \ref{h(a)-upp} and (\ref{mean-jun}), we obtain
\begin{align}\label{mean-m2}
           \int_{\mathfrak{m}_2}\mathcal{K}(\alpha)\mathrm{d}\alpha
    \ll &  \,\,\sup_{\alpha\in\mathfrak{m}_2}|h(\alpha)|\times \int_0^1 \big|f_2(\alpha)f_3^2(\alpha)f_3^*(\alpha)f_k^*(\alpha)\big|\mathrm{d}\alpha
                 \nonumber \\
    \ll &  \,\,n^{\frac{2}{9}+\frac{5}{12k}-24\varepsilon}\cdot n^{\frac{13}{18}+\frac{5}{12k}+\varepsilon}\ll
               n^{\frac{17}{18}+\frac{5}{6k}-23\varepsilon}.
\end{align}
For $\alpha\in\mathfrak{m}_1$, it follows from Theorem 4.1 in Vaughan \cite{Vaughan-book} that
\begin{equation}\label{h=W+E}
    h(\alpha)=\mathcal{W}(\alpha)+O\big(DQ_1^{\frac{1}{2}+\varepsilon}\big)
             =\mathcal{W}(\alpha)+O\big(n^{\frac{17}{72}+\frac{5}{24k}-25\varepsilon}\big),
\end{equation}
where $\mathcal{W}(\alpha)$ is defined by (\ref{W(alpha)-def}). Define
 \begin{equation*}
    \mathcal{K}_1(\alpha)=\mathcal{W}(\alpha)f_2(\alpha)f_3^2(\alpha)f_3^*(\alpha)f_k^*(\alpha)e(-n\alpha).
 \end{equation*}
 Then by (\ref{mean-jun}) and (\ref{h=W+E}) we have
\begin{equation}\label{K=K_1+E}
   \int_{\mathfrak{m}_1}\mathcal{K}(\alpha)\mathrm{d}\alpha
   =\int_{\mathfrak{m}_1}\mathcal{K}_1(\alpha)\mathrm{d}\alpha+O\big(n^{\frac{23}{24}+\frac{5}{8k}-24\varepsilon}\big).
\end{equation}
 Let
\begin{equation*}
   \mathcal{M}_0(q,a)=\bigg(\frac{a}{q}-\frac{1}{n^{\frac{25}{36}+\frac{5}{12k}}},\frac{a}{q}+\frac{1}{n^{\frac{25}{36}+\frac{5}{12k}}}\bigg],\qquad
  \mathcal{M}_0=\bigcup_{1\leqslant q\leqslant Q_0}\bigcup_{\substack{a=-q\\(a,q)=1}}^{2q}\mathcal{M}_0(q,a),
\end{equation*}
\begin{equation*}
   \mathcal{M}_1(q,a)=\mathcal{M}(q,a)\setminus\mathcal{M}_0(q,a),\qquad
    \mathcal{M}_1=\bigcup_{1\leqslant q\leqslant Q_0}\bigcup_{\substack{a=-q\\(a,q)=1}}^{2q}\mathcal{M}_1(q,a),
\end{equation*}
\begin{equation*}
    \mathcal{M}=\bigcup_{1\leqslant q\leqslant Q_0}\bigcup_{\substack{a=-q\\(a,q)=1}}^{2q}\mathcal{M}(q,a),
\end{equation*}
where $\mathcal{M}(q,a)$ is defined by (\ref{M(q,a)-def}). Then we have $\mathfrak{m}_1\subseteq\mathcal{I}_0\subseteq\mathcal{M}$.
By Dirichlet's theorem on Diophantine rational approximation, we obtain
\begin{align}\label{K_1-upper=1+2}
           \int_{\mathfrak{m}_1}\mathcal{K}_1(\alpha)\mathrm{d}\alpha
  \ll &   \,\, \sum_{1\leqslant q\leqslant Q_0}
           \sum_{\substack{a=-q\\(a,q)=1}}^{2q}\int_{\mathfrak{m}_1\cap\mathcal{M}_0(q,a)}|\mathcal{K}_1(\alpha)|\mathrm{d}\alpha
                    \nonumber \\
    &  \,\,+\sum_{1\leqslant q\leqslant Q_0}
    \sum_{\substack{a=-q\\(a,q)=1}}^{2q}\int_{\mathfrak{m}_1\cap\mathcal{M}_1(q,a)}|\mathcal{K}_1(\alpha)|\mathrm{d}\alpha .
\end{align}
By Lemma 4.2 of Titchmarsh \cite{Titchmarsh-book}, we have
\begin{equation*}
    w_j(\lambda)\ll \frac{X_j}{1+|\lambda|n},
\end{equation*}
from which and the trivial estimate $(q,d^2)\leqslant(q,d)^2$, we deduce that
\begin{align} \label{W(a)-upper}
   |\mathcal{W}(\alpha)|  \ll & \,\, \sum_{d\leqslant D}\frac{\tau(d)}{d}(q,d^2)^{1/2}q^{-1/2}|w_2(\lambda)|
                         \nonumber \\
    \ll & \,\,\tau_3(q)q^{-1/2}|w_2(\lambda)|\log^2n \ll \frac{\tau_3(q)X_2\log^2n}{q^{1/2}(1+|\lambda|n)}.
\end{align}
 Therefore, for $\alpha\in\mathcal{M}_1(q,a)$, we get
\begin{equation*}
 \mathcal{W}(\alpha)\ll n^{\frac{7}{36}+\frac{5}{12k}}\log^2n,
\end{equation*}
which combines (\ref{mean-jun}) to derive that
\begin{align}\label{K_1-upper-M_1}
      & \,\, \sum_{1\leqslant q\leqslant Q_0}
      \sum_{\substack{a=-q\\(a,q)=1}}^{2q}\int_{\mathfrak{m}_1\cap\mathcal{M}_1(q,a)}|\mathcal{K}_1(\alpha)|\mathrm{d}\alpha
                   \nonumber \\
      \ll & \,\, n^{\frac{7}{36}+\frac{5}{12k}}\log^2n
                 \times\int_0^1\big|f_2(\alpha)f_3^2(\alpha)f_3^*(\alpha)f_k^*(\alpha)\big|\mathrm{d}\alpha
              \ll n^{\frac{11}{12}+\frac{5}{6k}+\varepsilon}.
\end{align}
For $\alpha\in\mathcal{M}_0(q,a)$, it follows from Lemma 4.8 of Titchmarsh \cite{Titchmarsh-book} that
\begin{equation*}
   f_3(\alpha)=\Delta_3(\alpha)+\mathcal{V}_3(\alpha)+O(1).
\end{equation*}
Hence, one obtain
\begin{align}\label{I-123}
         & \,\, \sum_{1\leqslant q\leqslant Q_0}
      \sum_{\substack{a=-q\\(a,q)=1}}^{2q}\int_{\mathfrak{m}_1\cap\mathcal{M}_0(q,a)}|\mathcal{K}_1(\alpha)|\mathrm{d}\alpha
                          \nonumber \\
      \ll & \,\, \sum_{1\leqslant q\leqslant Q_0}
                \sum_{\substack{a=-q\\(a,q)=1}}^{2q}\int_{\mathfrak{m}_1\cap\mathcal{M}_0(q,a)}
                 \big|\mathcal{W}(\alpha)\Delta_3(\alpha)f_2(\alpha)f_3(\alpha)f_3^*(\alpha)f_k^*(\alpha)\big|\mathrm{d}\alpha
                           \nonumber \\
         & \,\,+\sum_{1\leqslant q\leqslant Q_0}\sum_{\substack{a=-q\\(a,q)=1}}^{2q}\int_{\mathfrak{m}_1\cap\mathcal{M}_0(q,a)}
            \big|\mathcal{W}(\alpha)\mathcal{V}_3(\alpha)f_2(\alpha)f_3(\alpha)f_3^*(\alpha)f_k^*(\alpha)\big|\mathrm{d}\alpha
                        \nonumber \\
         & \,\, + \sum_{1\leqslant q\leqslant Q_0}
                  \sum_{\substack{a=-q\\(a,q)=1}}^{2q}\int_{\mathfrak{m}_1\cap\mathcal{M}_0(q,a)}
                  \big|\mathcal{W}(\alpha)f_2(\alpha)f_3(\alpha)f_3^*(\alpha)f_k^*(\alpha)\big|\mathrm{d}\alpha
                          \nonumber \\
     =: & \,\,\,\, \mathfrak{I}_1+\mathfrak{I}_2+\mathfrak{I}_3,
\end{align}
where $\Delta_3(\alpha)$ and $\mathcal{V}_3(\alpha)$ are defined by (\ref{Delta_k-def}) and (\ref{V_k-def}), respectively.

It follows from Cauchy's inequality,  Lemma \ref{mean-value} and Lemma \ref{W(a)-Delta-k} that
\begin{align}\label{I_1-upper}
   \mathfrak{I}_1  \ll &  \,\,\Bigg(\sum_{1\leqslant q\leqslant Q_0}\sum_{\substack{a=-q\\(a,q)=1}}^{2q}\int_{\mathcal{M}(q,a)}
                \big|\mathcal{W}(\alpha)\Delta_3(\alpha)\big|^2\mathrm{d}\alpha\Bigg)^{\frac{1}{2}}
                \Bigg(\int_0^1\big|f_2(\alpha)f_3(\alpha)f_3^*(\alpha)f_k^*(\alpha)\big|^2\mathrm{d}\alpha\Bigg)^{\frac{1}{2}}
                    \nonumber   \\
        \ll &  \,\, \big(n^{\frac{2}{3}}\log^{-100A}n\big)^{\frac{1}{2}}\big(n^{\frac{11}{9}+\frac{5}{3k}}\log^{c+8}n\big)^{\frac{1}{2}}
                    \ll n^{\frac{17}{18}+\frac{5}{6k}}\log^{-40A}n.
\end{align}
By (\ref{W(a)-upper}), it is easy to see that, for $\alpha\in\mathfrak{m}_1$, there holds
\begin{equation}\label{W(a)-upper-m_1}
  \sup_{\alpha\in\mathfrak{m}_1} |\mathcal{W}(\alpha)|\ll n^{\frac{1}{2}}\log^{-30A}n.
\end{equation}
Therefore, by Lemma \ref{mean-value}, Lemma \ref{two-mean}, (\ref{W(a)-upper-m_1}) and Cauchy's inequality, we derive that
\begin{align}\label{I_2-upper}
   \mathfrak{I}_2  \ll &  \,\,\sup_{\alpha\in\mathfrak{m}_1} |\mathcal{W}(\alpha)| \cdot \Bigg(\sum_{1\leqslant q\leqslant Q_0}
                           \sum_{\substack{a=-q\\(a,q)=1}}^{2q}\int_{\mathcal{M}(q,a)}
                           \big|\mathcal{V}_3(\alpha)\big|^2\mathrm{d}\alpha\Bigg)^{\frac{1}{2}}
                                      \nonumber   \\
& \qquad \times\Bigg(\int_0^1\big|f_2(\alpha)f_3(\alpha)f_3^*(\alpha)f_k^*(\alpha)\big|^2\mathrm{d}\alpha\Bigg)^{\frac{1}{2}}
                                     \nonumber   \\
   \ll &  \,\, \big(n^{\frac{1}{2}}\log^{-30A}n\big) \cdot \big(n^{-\frac{1}{3}}\log^{21A}n\big)^{\frac{1}{2}} \cdot
                             \big(n^{\frac{11}{9}+\frac{5}{3k}}\log^{c+8}n\big)^{\frac{1}{2}}
                      \nonumber   \\
        \ll &   \,\, n^{\frac{17}{18}+\frac{5}{6k}}\log^{-5A}n.
\end{align}
It follows from Lemma \ref{mean-value} and Lemma \ref{two-mean} that
\begin{align}\label{I_3-upper}
   \mathfrak{I}_3 \ll & \,\,  \Bigg(\sum_{1\leqslant q\leqslant Q_0}
      \sum_{\substack{a=-q\\(a,q)=1}}^{2q}\int_{\mathcal{M}(q,a)}
      \big|\mathcal{W}(\alpha)\big|^2\mathrm{d}\alpha\Bigg)^{\frac{1}{2}}
      \Bigg(\int_0^1\big|f_2(\alpha)f_3(\alpha)f_3^*(\alpha)f_k^*(\alpha)\big|^2\mathrm{d}\alpha\Bigg)^{\frac{1}{2}}
                             \nonumber   \\
            \ll & \,\,  (\log^{21A}n)^{\frac{1}{2}}\cdot(n^{\frac{11}{9}+\frac{5}{3k}}\log^{c+8}n)^{\frac{1}{2}}
                         \ll n^{\frac{11}{18}+\frac{5}{6k}+\varepsilon} \ll n^{\frac{17}{18}+\frac{5}{6k}-\varepsilon}.
\end{align}
Combining (\ref{I-123}), (\ref{I_1-upper}), (\ref{I_2-upper}) and (\ref{I_3-upper}), we can deduce that
\begin{equation}\label{K_1-upper-1}
\sum_{1\leqslant q\leqslant Q_0}\sum_{\substack{a=-q\\(a,q)=1}}^{2q}\int_{\mathfrak{m}_1\cap\mathcal{M}_0(q,a)}|\mathcal{K}_1(\alpha)|\mathrm{d}\alpha
\ll n^{\frac{17}{18}+\frac{5}{6k}}\log^{-5A}n.
\end{equation}
From (\ref{K=K_1+E}), (\ref{K_1-upper=1+2}), (\ref{K_1-upper-M_1}) and (\ref{K_1-upper-1}), we deduce that
\begin{equation}\label{K(a)-upper-m1}
  \int_{\mathfrak{m}_1}\mathcal{K}(\alpha)\mathrm{d}\alpha\ll  n^{\frac{17}{18}+\frac{5}{6k}}\log^{-5A}n.
\end{equation}
Similarly, we obtain
\begin{equation}\label{K(a)-upper-m0}
  \int_{\mathfrak{m}_0}\mathcal{K}(\alpha)\mathrm{d}\alpha\ll  n^{\frac{17}{18}+\frac{5}{6k}}\log^{-5A}n.
\end{equation}
For $\alpha\in\mathscr{M}_0$, define
\begin{equation*}
  \mathcal{K}_0(\alpha)=\mathcal{W}(\alpha)\mathcal{V}_2(\alpha)\mathcal{V}_3^2(\alpha)
                        \mathcal{V}_3^*(\alpha)\mathcal{V}_k^*(\alpha)e(-n\alpha).
\end{equation*}
By noticing that (\ref{h=W+E}) still holds for $\alpha\in\mathscr{M}_0$, it follows from Lemma \ref{substi-lemma} and (\ref{h=W+E}) that
\begin{equation*}
  \mathcal{K}(\alpha)-\mathcal{K}_0(\alpha)\ll n^{\frac{35}{18}+\frac{5}{6k}}\exp\big(-\log^{1/4}n\big),
\end{equation*}
which implies that
\begin{equation}\label{K(a)=K_0(a)+error}
   \int_{\mathscr{M}_0}\mathcal{K}(\alpha)\mathrm{d}\alpha = \int_{\mathscr{M}_0}\mathcal{K}_0(\alpha)\mathrm{d}\alpha
    +O\big( n^{\frac{17}{18}+\frac{5}{6k}}\log^{-A}n\big).
\end{equation}
By the well--known standard technique  in the Hardy--Littlewood method, we deduce that
\begin{equation}\label{K_0=junzhi}
    \int_{\mathscr{M}_0}K_0(\alpha)\mathrm{d}\alpha =
    \sum_{m\leqslant D^{2/3}}a(m)\sum_{\ell\leqslant D^{1/3}}b(\ell)\frac{\mathfrak{S}_{m\ell}(n)}{m\ell}\mathcal{J}(n)
    +O\big( n^{\frac{17}{18}+\frac{5}{6k}}\log^{-A}n\big),
\end{equation}
and
\begin{equation}\label{J-singular-order}
    \mathcal{J}(n)\asymp n^{\frac{17}{18}+\frac{5}{6k}}.
\end{equation}
Finally, Proposition \ref{mean-value-1} follows from  (\ref{mean-fenjie}), (\ref{mean-m2}) and (\ref{K(a)-upper-m1})--(\ref{J-singular-order}). This completes the proof of Proposition \ref{mean-value-1}.
\end{proof}

By the same method, we have the following Proposition.
\begin{proposition}\label{mean-value-2}
For $3\leqslant k\leqslant14$, define
\begin{equation*}
  J_r(n,d)=\sum_{\substack{(\ell p)^2+m^2+p_2^3+p_3^3+p_4^3+p_5^k=n \\ X_2<\ell p\leqslant2X_2, \,\, \ell\in\mathcal{N}_r,  \,\,
   m\equiv0 \!\!\!\! \pmod d\\ X_3<p_2,\,p_3\leqslant2X_3,\,\, X_3^*<p_4\leqslant2X_3^* \\ X_k^*<p_5\leqslant2X_k^*}}
  \left(\frac{\log p}{\log \displaystyle\frac{X_2}{\ell}}\prod_{j=2}^5\log p_j\right).
\end{equation*}
Then we have
\begin{equation*}
  \sum_{m\leqslant D^{2/3}}a(m)\sum_{t\leqslant D^{1/3}}b(t)\bigg(J_r(n,mt)-\frac{c_r(k)\mathfrak{S}_{mt}(n)}{mt\log X_2}
  \mathcal{J}(n)\bigg)\ll n^{\frac{17}{18}+\frac{5}{6k}}\log^{-A}n,
\end{equation*}
where $c_r(k)$ is defined by (\ref{c_r-def}).
\end{proposition}

\section{On the function $\omega(d)$}

In this section, we shall investigate the function $\omega(d)$ which is defined in (\ref{omega(d)-def}) and required in the proof of  Theorem \ref{theorem-mixed}.
\begin{lemma}\label{congruence-lemma}
  For $3\leqslant k\leqslant14$, let $\mathscr{K}(q,n)$ and $\mathscr{L}(q,n)$ denote the number of solutions of the congruences
\begin{equation*}
   x^2+u_1^3+u_2^3+u_3^3+u_4^k\equiv n \!\!\!\!\pmod q, \quad 1\leqslant x,u_j\leqslant q,\quad (xu_j,q)=1 ,
\end{equation*}
 and
\begin{equation*}
     x_1^2+x_2^2+u_1^3+u_2^3+u_3^3+u_4^k\equiv n \!\!\!\!\pmod q,\quad 1\leqslant x_i,u_j\leqslant q, \quad (x_2u_j,q)=1 ,
\end{equation*}
 respectively. Then, for all $n\equiv0 \pmod 2$, we have $\mathscr{L}(p,n)>\mathscr{K}(p,n)$ for all primes. Moreover, there holds
\begin{equation*}
    \mathscr{L}(p,n)=p^5+O(p^4),
\end{equation*}
\begin{equation*}
    \mathscr{K}(p,n)=p^4+O(p^3).
\end{equation*}
\end{lemma}
\begin{proof}
   Let $\mathscr{L}^*(q,n)$ denote the number of solutions of the congruence
\begin{equation*}
    x_1^2+x_2^2+u_1^3+u_2^3+u_3^3+u_4^k\equiv n \!\!\!\!\pmod q,\quad 1\leqslant x_i,u_j\leqslant q, \quad (x_1x_2u_j,q)=1.
\end{equation*}
Then by the orthogonality of Dirichlet characters, we have
\begin{align} \label{L*(p,N)=p-1+E_p}
    p\cdot\mathscr{L}^*(p,n)  = & \,\,\sum_{a=1}^p S_2^{*2}(p,a)S_3^{*3}(p,a)S_k^{*}(p,a)e\Big(-\frac{an}{p}\Big)
                        \nonumber \\
     = & \,\, (p-1)^6+E_p,
\end{align}
where
\begin{equation*}
    E_p=\sum_{a=1}^{p-1} S_2^{*2}(p,a)S_3^{*3}(p,a)S_k^{*}(p,a)e\Big(-\frac{an}{p}\Big).
\end{equation*}
By (iv) of Lemma \ref{Hua-compo}, we have
\begin{equation}\label{E_p-upper}
   |E_p|\leqslant (p-1)(\sqrt{p}+1)^2(2\sqrt{p}+1)^3(13\sqrt{p}+1).
\end{equation}
It is easy to check that $|E_p|<(p-1)^6$ for $p\geqslant19$. Hence we get $\mathscr{L}^*(p,n)>0$ for $p\geqslant19$. On the other hand, for $p=2,3,5,7,11,13,17$, we can check $\mathscr{L}^*(p,n)>0$ directly by hand. Therefore, we obtain $\mathscr{L}^*(p,n)>0$ for all primes and
\begin{equation}\label{L(p,N)=L^*+K(p,N)}
  \mathscr{L}(p,n)=\mathscr{L}^*(p,n)+\mathscr{K}(p,n)>\mathscr{K}(p,n).
\end{equation}
 From (\ref{L*(p,N)=p-1+E_p}) and (\ref{E_p-upper}), we derive that
\begin{equation}\label{L^*(p,N)-asymp}
   \mathscr{L}^*(p,n)=p^5+O(p^4).
\end{equation}
By a similar argument of (\ref{L*(p,N)=p-1+E_p}) and (\ref{E_p-upper}), we have
\begin{equation}\label{K(p,N)-asymp}
   \mathscr{K}(p,n)=p^4+O(p^3).
\end{equation}
Combining (\ref{L(p,N)=L^*+K(p,N)})--(\ref{K(p,N)-asymp}), we obtain the desired results.
\end{proof}

\begin{lemma}\label{S(N)-convergence}
   The series $\mathfrak{S}(n)$ is convergent and satisfying $\mathfrak{S}(n)>0$.
\end{lemma}
\begin{proof}
   From (i) and (ii) of Lemma \ref{Hua-compo}, we obtain
\begin{equation*}
   |A(q,n)|\ll \frac{|B(q,n)|}{q\varphi^5(q)}\ll \frac{q^{2+5\varepsilon}}{\varphi^4(q)} \ll\frac{q^{2+5\varepsilon}(\log\log q)^4}{q^4}\ll\frac{1}{q^{3/2}}.
\end{equation*}
Thus, the series
\begin{equation*}
   \mathfrak{S}(n)=\sum_{q=1}^\infty A(q,n)
\end{equation*}
converges absolutely. Noting the fact that $A(q,n)$ is multiplicative in $q$ and by (v) of Lemma \ref{Hua-compo}, we get
\begin{equation}\label{S(N)-prod}
   \mathfrak{S}(n)=\prod_{p}\big(1+A(p,n)\big).
\end{equation}
From (iii) and (iv) of Lemma \ref{Hua-compo}, we know that, for $p\geqslant29$, we have
\begin{equation*}
   |A(p,n)|\leqslant\frac{(p-1)\sqrt{p}(\sqrt{p}+1)(2\sqrt{p}+1)^3(13\sqrt{p}+1)}{p(p-1)^5}\leqslant\frac{200}{p^2}.
\end{equation*}
Therefore, there holds
\begin{equation}\label{S(N)-wei}
  \prod_{p\geqslant29}\big(1+A(p,n)\big)\geqslant\prod_{p\geqslant29}\bigg(1-\frac{200}{p^2}\bigg)\geqslant c_1>0.
\end{equation}
On the other hand, it is easy to see that
\begin{equation}\label{S(N)-yinzi}
  1+A(p,n)=\frac{\mathscr{L}(p,n)}{(p-1)^5}.
\end{equation}
 By Lemma \ref{congruence-lemma}, we have $\mathscr{L}(p,n)>0$ for all $p$ with $n\equiv0 \pmod 2$, and thus $1+A(p,n)>0$. Consequently, we obtain
\begin{equation}\label{S(N)-head}
  \prod_{p<29}\big(1+A(p,n)\big)\geqslant c_2>0.
\end{equation}
Combining (\ref{S(N)-prod}), (\ref{S(N)-wei}) and (\ref{S(N)-head}), we conclude that $\mathfrak{S}(n)>0$,
which completes the proof of Lemma \ref{S(N)-convergence}.
\end{proof}

In view of Lemma \ref{S(N)-convergence}, we define
\begin{equation}\label{omega(d)-def}
   \omega(d)=\frac{\mathfrak{S}_d(n)}{\mathfrak{S}(n)}.
\end{equation}
Similar to (\ref{S(N)-prod}), we have
\begin{equation}\label{Sd(N)-prod}
   \mathfrak{S}_d(n)=\prod_{p}\big(1+A_d(p,n)\big).
\end{equation}
If $(d,q)=1$, then we have $S_k(q,ad^k)=S_k(q,a)$. Moreover, if $p|d$, then we get $A_d(p,n)=A_p(p,n)$.
Therefore, we derive that
\begin{equation}\label{omega-frac}
   \omega(p)=\frac{1+A_p(p,n)}{1+A(p,n)},\qquad \omega(d)=\prod_{p|d}\omega(p).
\end{equation}
Also, it is easy to show that
\begin{equation}\label{omega-frac-up}
   1+A_p(p,n)=\frac{p}{(p-1)^5}\mathscr{K}(p,n).
\end{equation}
Using (\ref{S(N)-yinzi}), (\ref{omega-frac}) and (\ref{omega-frac-up}), we derive
\begin{equation*}
   \omega(p)=\frac{p\cdot\mathscr{K}(p,n)}{\mathscr{L}(p,n)},
\end{equation*}
from which and Lemma \ref{congruence-lemma}, we derive the following lemma.

\begin{lemma}\label{omega(p)-property}
   The function $\omega(d)$ is multiplicative and satisfies
\begin{equation}\label{Omega-condition}
   0\leqslant\omega(p)<p,\qquad \omega(p)=1+O(p^{-1}).
\end{equation}
\end{lemma}

\section{Proof of Theorem \ref{theorem-mixed}}

In this section, let $f(s)$ and $F(s)$ denote the classical functions in the linear sieve theory. Then it follows from (2.8) and (2.9) of Chapter 8 in \cite{Halberstam-Richert} that
\begin{equation*}
  F(s)=\frac{2e^\gamma}{s},\quad 1\leqslant s\leqslant3;\qquad f(s)=\frac{2e^\gamma\log(s-1)}{s},\quad 2\leqslant s\leqslant4.
\end{equation*}
In the proof of Theorem \ref{theorem-mixed}, let $\lambda^{\pm}(d)$ be the lower and upper bounds for Rosser's weights of level $D$, hence
for any positive integer $d$ we have
\begin{equation*}
  |\lambda^{\pm}(d)|\leqslant1,\quad \lambda^{\pm}(d)=0 \quad \textrm{if} \quad d>D \quad \textrm{or}\quad \mu(d)=0.
\end{equation*}
For further properties of Rosser's weights we refer to Iwaniec \cite{Iwaniec-1980-1}. Define
\begin{equation*}
  \mathscr{W}(z)=\prod_{2<p<z}\bigg(1-\frac{\omega(p)}{p}\bigg).
\end{equation*}
Then from Lemma \ref{omega(p)-property} and Mertens' prime number theorem (See \cite{Mertens-1874}) we obtain
\begin{equation}\label{V(z)-asymp-order}
    \mathscr{W}(z)\asymp \frac{1}{\log N}.
\end{equation}
In order to prove Theorem \ref{theorem-mixed}, we need the following lemma.
\begin{lemma}
  Under the condition (\ref{Omega-condition}), then if $z\leqslant D$, there holds
\begin{equation}\label{Rosser-lower}
   \sum_{d|\mathscr{P}}\frac{\lambda^-(d)\omega(d)}{d}\geqslant\mathscr{W}(z)\bigg(f\bigg(\frac{\log D}{\log z}\bigg)+O\big(\log^{-1/3}D\big)\bigg),
\end{equation}
and if $z\leqslant D^{1/2}$, there holds
\begin{equation}\label{Rosser-upper}
   \sum_{d|\mathscr{P}}\frac{\lambda^+(d)\omega(d)}{d}\leqslant\mathscr{W}(z)\bigg(F\bigg(\frac{\log D}{\log z}\bigg)+O\big(\log^{-1/3}D\big)\bigg).
\end{equation}
\end{lemma}
\begin{proof}
   See Iwaniec \cite{Iwaniec-1980-2}, (12) and (13) of Lemma 3.
\end{proof}

From the definition of $\mathcal{B}_r$, we know that $r\leqslant\big[\frac{36k}{15-k}\big]$. Hence we obtain
\begin{align}\label{R_k(N)-lower}
                    \mathscr{R}_k(N)
      \geqslant &  \,\,\sum_{\substack{m^2+p_1^2+p_2^3+p_3^3+p_4^3+p_5^k=n\\ X_2<m,p_1\leqslant2X_2,  \,\,
                   (m,\mathscr{P})=1\\ X_3<p_2,p_3\leqslant2X_3 \\ X_3^*<p_4\leqslant2X_3^* \\ X_k^*<p_5\leqslant2X_k^*}}1-
                    \sum_{r=r(k)+1}^{\big[\frac{36k}{15-k}\big]} \sum_{\substack{m^2+p_1^2+p_2^3+p_3^3+p_4^3+p_5^k=n\\ m\in\mathcal{B}_r,\,\,\,\, X_2<p_1\leqslant2X_2 \\
                    X_3<p_2,p_3\leqslant2X_3\\ X_3^*<p_4\leqslant2X_3^* \\ X_k^*<p_5\leqslant2X_k^* }}1
                           \nonumber \\
      =: & \,\,\,\Upsilon_0-\sum_{r=r(k)+1}^{\big[\frac{36k}{15-k}\big]}\Upsilon_{r}.
\end{align}
By the property (\ref{Rosser-lower}) of Rosser's weight $\lambda^-(d)$ and Proposition \ref{mean-value-1}, we get
\begin{align}\label{Upsilon_0-lower}
                 \Upsilon_0   \geqslant
  & \,\, \frac{1}{\log\bf{\Xi}}\sum_{\substack{m^2+p_1^2+p_2^3+p_3^3+p_4^3+p_5^k=n\\ X_2<m,p_1\leqslant2X_2,  \,\,
                   (m,\mathscr{P})=1\\ X_3<p_2,p_3\leqslant2X_3 \\ X_3^*<p_4\leqslant2X_3^* \\ X_k^*<p_5\leqslant2X_k^*}}  \prod_{j=1}^5\log p_j
                               \nonumber  \\
   = & \,\,       \frac{1}{\log\bf{\Xi}}\sum_{\substack{m^2+p_1^2+p_2^3+p_3^3+p_4^3+p_5^k=n\\
                  X_2<m,p_1\leqslant2X_2,  \,\, X_3^*<p_4\leqslant2X_3^* \\ X_3<p_2,p_3\leqslant2X_3, \,\,  X_k^*<p_5\leqslant2X_k^* }}
                  \Bigg(\prod_{j=1}^5\log p_j\Bigg) \sum_{d|(m,\mathscr{P})}\mu(d)
                                 \nonumber  \\
   \geqslant & \,\, \frac{1}{\log\bf{\Xi}}\sum_{\substack{m^2+p_1^2+p_2^3+p_3^3+p_4^3+p_5^k=n\\
                     X_2<m,p_1\leqslant2X_2,  \,\, X_3^*<p_4\leqslant2X_3^* \\ X_3<p_2,p_3\leqslant2X_3, \,\,  X_k^*<p_5\leqslant2X_k^*  }}
                    \Bigg(\prod_{j=1}^5\log p_j\Bigg) \sum_{d|(m,\mathscr{P})}\lambda^-(d)
                                  \nonumber  \\
          = & \,\,  \frac{1}{\log\bf{\Xi}} \sum_{d|\mathscr{P}}\lambda^-(d)J(n,d)
                                 \nonumber  \\
   = & \,\, \frac{1}{\log\bf{\Xi}} \sum_{d|\mathscr{P}}\frac{\lambda^-(d)\mathfrak{S}_d(n)}{d} \mathcal{J}(n)+ O\big(n^{\frac{17}{18}+\frac{5}{6k}}\log^{-A}n\big)
                                  \nonumber  \\
   = & \,\, \frac{1}{\log\bf{\Xi}} \Bigg(\sum_{d|\mathscr{P}}\frac{\lambda^-(d)\omega(d)}{d}\Bigg)\mathfrak{S}(n)\mathcal{J}(n)
             + O\big(n^{\frac{17}{18}+\frac{5}{6k}}\log^{-A}n\big)
                                  \nonumber  \\
   \geqslant & \,\, \frac{\mathfrak{S}(n)\mathcal{J}(n)\mathscr{W}(z)}{\log\bf{\Xi}} f(3)\Big(1+O\big(\log^{-1/3}D\big)\Big)
                     + O\big(n^{\frac{17}{18}+\frac{5}{6k}}\log^{-A}n\big) .
\end{align}

By the property (\ref{Rosser-upper}) of Rosser's weight $\lambda^+(d)$ and Proposition \ref{mean-value-2}, we have
\begin{align}\label{Upsilon_r-upper}
  \Upsilon_r   \leqslant & \,\, \sum_{\substack{(\ell p)^2+m^2+p_2^3+p_3^3+p_4^3+p_5^k=n\\ \ell\in\mathcal{N}_r,\,\,\,
               X_2<\ell p\leqslant2X_2,\,\,(m,\mathscr{P})=1 \\ X_3<p_2,p_3\leqslant2X_3\\ X_3^*<p_4\leqslant2X_3^* \\ X_k^*<p_5\leqslant2X_k^*}}1
                                \nonumber  \\
    \leqslant & \,\, \frac{1}{\log\bf{\Theta}}\sum_{\substack{(\ell p)^2+m^2+p_2^3+p_3^3+p_4^3+p_5^k=n\\
                \ell\in\mathcal{N}_r,\,\,\, X_2<\ell p\leqslant2X_2,\,\,(m,\mathscr{P})=1 \\ X_3<p_2,p_3\leqslant2X_3\\ X_3^*<p_4\leqslant2X_3^* \\ X_k^*<p_5\leqslant2X_k^*}}
                     \frac{\log p}{\log\frac{X_2}{\ell}} \prod_{j=2}^5\log p_j
                                \nonumber  \\
    = & \,\,  \frac{1}{\log\bf{\Theta}} \sum_{\substack{(\ell p)^2+m^2+p_2^3+p_3^3+p_4^3+p_5^k=n\\
                \ell\in\mathcal{N}_r,\,\,\, X_2<\ell p\leqslant2X_2 \\ X_3<p_2,p_3\leqslant2X_3\\ X_3^*<p_4\leqslant2X_3^* \\ X_k^*<p_5\leqslant2X_k^*}}
                      \Bigg(\frac{\log p}{\log\frac{X_2}{\ell}} \prod_{j=2}^5\log p_j\Bigg)\sum_{d|(m,\mathscr{P})}\mu(d)
                                \nonumber  \\
    \leqslant & \,\,  \frac{1}{\log\bf{\Theta}} \sum_{\substack{(\ell p)^2+m^2+p_2^3+p_3^3+p_4^3+p_5^k=n\\
                \ell\in\mathcal{N}_r,\,\,\, X_2<\ell p\leqslant2X_2 \\ X_3<p_2,p_3\leqslant2X_3\\ X_3^*<p_4\leqslant2X_3^* \\ X_k^*<p_5\leqslant2X_k^*}}
                      \Bigg(\frac{\log p}{\log\frac{X_2}{\ell}} \prod_{j=2}^5\log p_j\Bigg)\sum_{d|(m,\mathscr{P})}\lambda^+(d)
                                \nonumber  \\
            = & \,\, \frac{1}{\log\bf{\Theta}} \sum_{d|\mathscr{P}} \lambda^+(d)J_r(n,d)
                                \nonumber  \\
            = & \,\, \frac{1}{\log\bf{\Theta}} \sum_{d|\mathscr{P}}
                     \frac{\lambda^+(d)c_r(k)\mathfrak{S}_d(n)}{d\log X_2}\mathcal{J}(n)
                      +O\big(n^{\frac{17}{18}+\frac{5}{6k}}\log^{-A}n\big)
                                \nonumber  \\
            = & \,\, \frac{c_r(k)\mathfrak{S}(n)\mathcal{J}(n)}{(\log X_2)\log\bf{\Theta}}
                     \sum_{d|\mathscr{P}}\frac{\lambda^+(d)\omega(d)}{d}+O\big(n^{\frac{17}{18}+\frac{5}{6k}}\log^{-A}n\big)
                                \nonumber  \\
    \leqslant & \,\, \frac{c_r(k)\mathfrak{S}(n)\mathcal{J}(n)\mathscr{W}(z)}{\log\bf{\Xi}} F(3)\Big(1+O\big(\log^{-1/3}D\big)\Big)
                      + O\big(n^{\frac{17}{18}+\frac{5}{6k}}\log^{-A}n\big).
\end{align}
According to simple numerical calculations, we know that
\begin{align*}
   &  c_4(3)\leqslant0.4443636,\,\, c_5(3)\leqslant0.0578256,\,\, c_j(3)\leqslant0.0027627 \,\,\,\, \textrm{with} \,\,\, 6\leqslant j\leqslant9;  \\
   &  c_5(4)\leqslant0.3029445,\,\, c_6(4)\leqslant0.0459743,\,\, c_j(4)\leqslant0.00388094\,\,\,\, \textrm{with} \,\,\,7\leqslant j\leqslant13;  \\
   &  c_6(5)\leqslant0.1892887,\,\, c_j(5)\leqslant0.0307123 \,\,\,\, \textrm{with} \,\,\, 7\leqslant j\leqslant 18;  \\
   &  c_6(6)\leqslant0.4867818,\,\, c_7(6)\leqslant0.1133016,\,\, c_8(6)\leqslant0.01913692,  \\
   &  c_j(6)\leqslant 0.00237244 \,\,\,\, \textrm{with} \,\,\, 9\leqslant j\leqslant 24; \\
   &  c_7(7)\leqslant 0.2978111,\,\, c_8(7)\leqslant0.0672273,\,\, c_j(7)\leqslant0.0117295 \,\,\,\, \textrm{with} \,\,\,
            9\leqslant j\leqslant 31;  \\
   &  c_8(8)\leqslant 0.1830229,\,\, c_9(8)\leqslant0.0407894,\,\,c_j(8)\leqslant0.0073521 \,\,\,\, \textrm{with} \,\,\,
            10\leqslant j\leqslant 41;  \\
   &  c_8(9)\leqslant 0.4323101,\,\, c_9(9)\leqslant0.1169923,\,\, c_{10}(9)\leqslant0.02614497,   \\
   &   c_{11}(9)\leqslant0.0048887,\,\, c_j(9)\leqslant0.000772739 \,\,\,\, \textrm{with} \,\,\, 12\leqslant j\leqslant 54; \\
   & c_9(10)\leqslant0.3023038,\,\,c_{10}(10)\leqslant0.0809431,\,\,c_{11}(10)\leqslant0.0184125,  \\
   & c_j(10)\leqslant0.003597861 \,\,\,\, \textrm{with} \,\,\, 12\leqslant j\leqslant 72;  \\
   & c_{10}(11)\leqslant0.2360241,\,\,c_{11}(11)\leqslant0.0639155,\,\, c_{12}(11)\leqslant0.01504156, \\
   & c_j(11)\leqslant0.003105002 \,\,\,\, \textrm{with} \,\,\, 13\leqslant j\leqslant 99;  \\
   & c_{11}(12)\leqslant0.2231261,\,\,c_{12}(12)\leqslant0.06262236,\,\,c_{13}(12)\leqslant 0.01555779, \\
   & c_{14}(12)\leqslant0.00344782,\,\, c_j(12)\leqslant0.0006868855 \,\,\,\,\textrm{with}\,\,\,15\leqslant j\leqslant 144;  \\
   & c_{12}(13)\leqslant0.2976851,\,\, c_{13}(13)\leqslant0.0895433,\,\, c_{14}(13)\leqslant0.0242215,\\
   & c_{15}(13)\leqslant0.005929363,\,\, c_{16}(13)\leqslant0.0013212887,\,\, \\
   &    c_{j}(13)\leqslant0.0002694412 \,\,\,\, \textrm{with} \,\,\, 17\leqslant j\leqslant 234;   \\
   & c_{14}(14)\leqslant0.2926583,\,\,c_{15}(14)\leqslant0.09172191,\,\,c_{16}(14)\leqslant 0.026363835, \\
   & c_{17}(14)\leqslant0.006978431,\,\, c_{18}(14)\leqslant0.001783123,\,\, \\
   &  c_j(14)\leqslant0.0002510648  \,\,\,\,\textrm{with}\,\,\,19\leqslant j\leqslant 504.
\end{align*}
Therefore, if we write
\begin{equation}\label{C_k-def}
 C(k)=\sum_{r=r(k)+1}^{\big[\frac{36k}{15-k}\big]}c_r(k),
\end{equation}
then we have
\begin{align}
 & \, C(3)<0.513241,\quad C(4)<0.376086,\quad C(5)<0.557837,\quad C(6)<0.657181,  \\
 & \, C(7)<0.634817,\quad C(8)<0.459081,\quad C(9)<0.613564,\quad C(10)<0.621131,  \\
 & \, C(11)<0.585117,\,\, C(12)<0.394051,\,\,\, C(13)<0.477439,\,\,\,\, C(14)<0.541523. \label{num-cal}
\end{align}
From (\ref{V(z)-asymp-order}), (\ref{R_k(N)-lower})--(\ref{num-cal}), we derive that
\begin{align*}
          \mathscr{R}_k(N) \geqslant
           & \,\,   \Bigg(f(3)-F(3)\sum_{r=r(k)+1}^{\big[\frac{36k}{15-k}\big]}c_r(k)\Bigg)\Big(1+O\big(\log^{-1/3}D\big)\Big)   \\
           & \,\,\, \times \frac{\mathfrak{S}(n)\mathcal{J}(n)\mathscr{W}(z)}{\log\bf{\Xi}} +  O\big(n^{\frac{17}{18}+\frac{5}{6k}}\log^{-A}n\big)  \\
\geqslant  & \,\,   \frac{2e^{\gamma}}{3} (\log2-0.657181)\Big(1+O\big(\log^{-1/3}D\big)\Big)    \\
           & \,\,\, \times   \frac{\mathfrak{S}(n)\mathcal{J}(n)\mathscr{W}(z)}{\log\bf{\Xi}} +  O\big(n^{\frac{17}{18}+\frac{5}{6k}}\log^{-A}n\big)  \\
       \gg & \,\,   n^{\frac{17}{18}+\frac{5}{6k}}\log^{-6}n,
\end{align*}
which completes the proof of Theorem \ref{theorem-mixed}.

\section*{Acknowledgement}

   The authors would like to express the most sincere gratitude to the referee
for his/her patience in refereeing this paper. This work is supported by the
Fundamental Research Funds for the Central Universities (Grant No. 2019QS02),
 and National Natural Science Foundation of China (Grant No. 11901566, 11971476).

\end{document}